\documentclass[reqno]{amsart}
\usepackage{amssymb}
\usepackage{amsthm}
\usepackage{bbm}
\usepackage{graphicx}

\usepackage[colorlinks=true]{hyperref}
\usepackage[all]{xypic}

\usepackage[final]{epsfig}
\usepackage{psfrag}
\usepackage{epstopdf}

\newtheorem{thm}{Theorem}[section]
\newtheorem{theorem}[thm]{Theorem}

\newtheorem{corollary}[thm]{Corollary}
\newtheorem{lemma}[thm]{Lemma}

\newtheorem{prop}[thm]{Proposition}
\theoremstyle{definition}

\theoremstyle{observation}

\theoremstyle{definition}

\newcommand{\M}{\mathcal{M}}

\newcommand{\defin}[1]{{\it #1}}
\newcommand{\R}{\mathbb{R}}
\newcommand{\N}{\mathbb{N}}
\newcommand{\Q}{\mathbb{Q}}

{\bf}{\it}
{\bf}{\it}
\newtheorem*{riemann-mapping-theorem}{Riemann Mapping  Theorem}{\bf}{\it}
{\bf}{\it}
{\bf}{\it}
{\bf}{\it}
{\bf}{\it}
{\bf}{\it}
{\bf}{\it}
{\bf}{\it}
{\bf}{\it}
{\bf}{\it}

\newenvironment{pf*}[1]{\proof[#1]}{\endproof}
\usepackage{euscript}

\newcommand{\cal}[1]{{\mathcal #1}}

\newcommand{\beq}{\begin{equation}}
\newcommand{\eeq}{\end{equation}}

\newcommand{\om}{\omega}

\newtheorem{defn}{Definition}[section]

\newcommand{\dist}{\operatorname{dist}}

\renewcommand{\mod}{\operatorname{mod}}

\newcommand{\wtl}{\widetilde}
\newcommand{\eps}{\epsilon}

\renewcommand{\Im}{\operatorname{Im}}

\numberwithin{equation}{section}

\newcommand{\propref}[1]{Proposition~\ref{#1}}

\newcommand{\cM}{{\mathcal M}}

\renewcommand{\cD}{{\cal D}}

\newcommand{\CC}{{\mathbb C}}
\newcommand{\RR}{{\mathbb R}}

\newcommand{\TT}{{\mathbb T}}
\newcommand{\ZZ}{{\mathbb Z}}
\newcommand{\NN}{{\mathbb N}}

\newcommand{\QQ}{{\mathbb Q}}

\newcommand{\ignore}[1]{{}}

    \title[Computational Complexity of real Julia sets]{Real quadratic Julia sets can have arbitrarily high complexity}
\author{Cristobal Rojas}
\address{Departamento de Matem\'aticas, Universidad Andres Bello, Chile.}
\email{crojas@mat-unab.cl} 

\author{Michael Yampolsky}
\address{Department of Mathematics, University of Toronto, Canada.}
\email{yampol@math.toronto.edu}
 \thanks{C.R was partially supported by projects FONDECYT  Regular 1190493 and Basal PFB-03 CMM-Universidad de Chile. M.Y. was partially supported by NSERC Discovery grant.}
\keywords{Complexity lower bounds, unimodal maps, renormalization.}
\subjclass[2010]{68Q17 and 37E05.} 
\begin{document}
\maketitle

\begin{abstract}  We show that there exist real parameters  $c\in (-2,0)$ for which the Julia set $J_{c}$ of the quadratic map $z^{2} + c$  has arbitrarily high computational complexity. More precisely, we show that for any given complexity threshold $T(n)$, there exist a real parameter $c$ such that the computational complexity of computing $J_{c}$ with $n$ bits of precision is higher than $T(n)$.  This is the first known class of real parameters with a non poly-time computable Julia set. 
\end{abstract}


\section{Introduction}

While the computability theory of polynomial Julia sets appears rather complete, the study of computational complexity of computable Julia sets offers many unanswered questions. Let us briefly overview the known results. In all of them the Julia set of a rational function $R$ is computed by a Turing Machine $M^\phi$ with an oracle for the coefficients of $R$. The first complexity result in this direction is independently due to Braverman \cite{thesis} and Rettinger \cite{Ret}, who showed that hyperbolic Julia sets have poly-time complexity. 


We note that the poly-time algorithm described in their results has been known to practitioners as {\it Milnor's Distance Estimator} \cite{Milnor-89}. Specializing to the quadratic family $f_c$, we note that Distance Estimator becomes very slow (exp-time) for the values of $c$ for which $f_c$ has a parabolic periodic point. This would appear to be a natural class of examples to look for a lower complexity bound. However, surprisingly, Braverman \cite{braverman-parabolic} proved that parabolic Julia sets are also polynomial-time computable. 
The algorithm presented in \cite{braverman-parabolic} is again explicit, and easy to implement in practice -- it is a refinement of Distance Estimator.

On the other hand, Binder, Braverman, and Yampolsky \cite{BBY2} proved that within the class of Siegel quadratics (the only case containing non computable Julia sets), there exists computable Julia sets whose time complexity can be arbitrarily high. 

A major open question is the complexity of quadratic Julia sets with Cremer points. They are notoriously hard to draw in practice; no high-resolution pictures have been produced to this day -- and although we know they are always computable, we do not know whether any of them are computably hard. 

Let us further specialize to real quadratic family $f_c$, $c\in\RR$. In this case, it was recently proved by Dudko and Yampolsky \cite{DudYam2} that almost every real quadratic Julia set is poly-time computable.  Conjecturally, the main technical result of \cite{DudYam2} should imply the same statement for complex parameters $c$ as well, but the conjecture in question (Collet-Eckmann parameters form a set of full measure among non-hyperbolic parameters) while long-established, is stronger than Density of Hyperbolicity Conjecture, which is the main open problem in the field.

It is also worth mentioning in this regard that the extreme non-hyperbolic examples in real dynamics are infinitely renormalizable quadratic polynomials. The archetypic such example is the celebrated Feigenbaum polynomial. In a different paper, Dudko and Yampolsky \cite{DudYam1} showed that the Feigenbaum Julia set also has  polynomial time complexity. 

The above theorems raise a natural question whether {\it all} real quadratic Julia sets are poly-time (the examples of \cite{BBY2} cannot have real values of $c$). In the present paper we answer this question in the negative by showing that

\begin{thm}There exists real parameters $c\in (-1.75,0)$ whose quadratic Julia sets have arbitrarily high computational complexity.  
\end{thm}


\section{Preliminaries}

 \subsection*{Computational Complexity of sets}
 
We give a very brief summary of relevant notions of Computability Theory and Computable Analysis. For a more in-depth
introduction, the reader is referred to e.g. \cite{BY-book}.
As is standard in Computer Science, we formalize the notion of
an algorithm as a {\it Turing Machine} \cite{Tur}.  
Let us begin by giving the modern definition of the notion of computable real
number,  which goes back to the seminal paper of Turing \cite{Tur}. By identifying $\Q$ with $\N$ through some effective enumeration, we can assume algorithms can operate on $\Q$. Then a real number $x\in\RR$ is called \defin{computable} if there is an algorithm  $M$ which, upon input $n$, halts and outputs a rational number $q_n$ such that  $|q_n-x|<2^{-n}$.
Algebraic numbers or  the familiar constants such as $\pi$, $e$, or the Feigenbaum constant  are computable real numbers. However, the set of all computable real numbers $\RR_C$ is necessarily countable, as there are only countably many Turing Machines. 

Computability of compact subsets of $\RR^k$ is defined by following the same principle.  Let us say that a point in $\RR^k$ is a {\it dyadic rational with denominator} $2^{-n}$ if it is of the form $\bar v\cdot 2^{-n}$, where $\bar v\in\ZZ^k$ and $n\in\NN$. 
Recall that {\it Hausdorff distance} between two
compact sets $K_1$, $K_2$ is
$$\dist_H(K_1,K_2)=\inf_\eps\{K_1\subset K^\eps_2\text{ and }K_2\subset K^\eps_1\},$$
where $$K^\eps=\bigcup_{z\in K}B(z,\eps)$$ stands for the $\eps$-neighbourhood of a set $K$.

We will also define the one-sided {\it distance from $K_1$ to $K_2$} as
$$\dist(K_1,K_2)=\inf_\eps\{K_1\subset K^\eps_2\},$$
so that
$$\dist_H(K_1,K_2)=\max(\dist(K_1,K_2),\dist(K_2,K_1)).$$

\begin{defn}\label{1}We say that a compact set $K\Subset\RR^k$ is {\it computable} if there exists an algorithm $M$ with a single input $n\in\NN$, which outputs a
finite set $C_n $ of dyadic rational points in $\RR^k$ such that
$$\dist_H(C_n,K)<2^{-n}.$$
\end{defn}

An equivalent way of defining computability, which is more convenient for discussing computational complexity is the following.
  For $\bar x=(x_1,\ldots,x_k)\in\RR^k$ let
the norm $||\bar x||_1$ be given by
$$||\bar x||_1=\max|x_i|.$$
\begin{defn}
  \label{def-local}
A compact set $K\Subset\RR^k$ is computable if there exists an algorithm $M$ which, given as input $(\bar v, n)$ representing a dyadic rational point $x$ in $\R^k$ whose coordinates have $n$ dyadic digits, outputs $0$ if $x$ is at distance strictly more than $2\cdot 2^{-n}$ from $K$ in $||\cdot||_1$ norm,  outputs $1$ if $x$ is at distance strictly less than $2^{-n}$ from $K$, and outputs either $0$ or $1$ in the ``borderline'' case.
\end{defn}

In the familiar context of $k=2$, such an algorithm can be used to ``zoom into'' the set $K$ on a computer screen with $W\times H$ square pixels to draw an accurate picture of a rectangular portion of $K$ of width $W\cdot 2^{-n}$ and height $H\cdot 2^{-n}$. $M$ decides which pixels in this picture have to be black (if their centers are $2^{-n}$-close to $K$) or white (if their centers are $2\cdot 2^{-n}$-far from $K$), allowing for some ambiguity in the intermediate case.

Let $C=\dist_H(K,0)$. 
For an algorithm $M$ as in Definition~\ref{def-local} let us denote by $T_{M}(n)$ the supremum of running times of $M$ over all dyadic points of size $n$ which are inside the ball of radius $2C$ centered at the origin: this is the computational cost of using $M$ for deciding the hardest pixel at the given resolution.

\begin{defn}
We say that a function $T:\NN\to\NN$ is a {\it lower bound} on time complexity of $K$ if for any $M$ as in Definition~\ref{def-local} there exists an infinite sequence $\{n_i\}$
such that
$$T_M(n_i)\geq T(n_i). $$
Similarly, we say that $T(n)$ is an {\it upper bound} on time complexity of $K$ if there exists an algorithm $M$ as in Definition~\ref{def-local} such that for all $n\in\NN$
$$T_M(n)\leq T(n).$$
\end{defn}

%

In this paper, we will be interested in the time complexity of Julia sets of  quadratic maps of the form $z^2 + c$, with $c\in \RR$.
As is standard in computing practice, we will assume that the algorithm can read the value of $c$ externally to produce a zoomed in picture of the Julia set. More formally, let us denote $\cD_n\subset \RR$ the set of dyadic rational numbers with denominator $2^{-n}$. 
We say that a function $\phi:\NN\to\QQ$ is an {\it oracle} for $c\in \RR$ if for every $m\in \NN$
$$\phi(m)\in \cD_m\text{ and }d(\phi(m),c)<2^{-(m-1)}.$$
We amend our definitions of computability and complexity of a compact set $K$ by allowing {\it oracle Turing Machines} $M^\phi$ where
$\phi$ is any function as above. 
On each step of the algorithm, $M^\phi$ may read the value of $\phi(m)$ for an arbitrary $m\in\NN$.

This approach allows us to separate the questions of computability and computational complexity of a parameter $c$ from that of the Julia set. It is crucial to note that reading the values of $\phi$ comes with a computational cost:

\medskip
\noindent
{\it querying $\phi$ with precision $m$ counts as $m$ time units. In other words, it takes
  $m$ ticks of the clock to read the first $m$ dyadic digits of $c$.
}

\medskip
\noindent
This is again in a full agreement with computing practice: to produce a verifiable picture of a set, we have to use the ``long arithmetic'' for constants, which are represented by sequences of dyadic bits. The computational cost grows with the precision of the computation, and manipulating a single bit takes one unit of machine time.


\subsection*{Julia sets of quadratic polynomials and the statement of the main result}

For a quadratic polynomial $f_{c}: z \mapsto z^{2} + c$, the filled-in Julia set $K_{c}$ of $f_{c}$ is defined as the set of points that do not escape under iteration of $f_{c}$:
$$
K_{c} = \{ z \in \CC: (f_{c}^{n}(z))_{n} \text{ is bounded} \},
$$ 
where $f_{c}^{n}$ denotes the $n^{th}$ iteration $f_{c}\circ f_{c} \circ \dots \circ f_{c}$ of $f_{c}$.  The Julia $J_{c}$ set of $f_{c}$ is $J_{c}=\partial K_{c}$. 

Our main result is the following. 

\bigskip
\noindent\textbf{Main Theorem.} \emph{
 Given any function $T:\NN\to\NN$, there exists a value of $c \in (-1.75,0)$ such that the map $P_c$ has a Julia set $J_{c}$ whose computational  complexity is bounded below by $T(n)$.
	}
\bigskip


\section{Proof of the Main Theorem}

\subsection{Parabolic implosion} 

 A point $\alpha$ is \emph{parabolic} for a complex quadratic map $f_c$  if $$f_c^k(\alpha)=\alpha \qquad  \text{ and } \qquad  (f_c^k)'(\alpha)= \exp(2\pi i m/n)$$ for some $k> 0, n > 0$ and $m\geq 0 $ with $m$ and $n$ relatively prime. The simplest possible example is $f_{1/4}(z)=z^2 + 1/4$, for which we have $\alpha = 1/2, k=1, m=0$.  The point $1/4$ is the cusp of the Mandelbrot set, the filled Julia set $K(f_{1/4})$ is a \emph{cauliflower} centered at $\om= 0 $, whose boundary $J(f_{1/4})$ is a Jordan curve.

 A parameter value $c$ is called \emph{super stable} if $0$ is periodic under $f_c$.  To each super stable parameter $c$ there corresponds a homeomorphic small copy $\mathcal{M}(c)$ of the Mandelbrot set $\mathcal{M}$ which contains $c$ and called \emph{the Mandelbrot set tuned by} $c$.  The \emph{root} of $\mathcal{M}(c)$ is the point corresponding to $1/4$ in $\mathcal{M}$, and the \emph{center} is $c$.  The root $r(c)$ of each little copy $\mathcal{M}(c)$ is a parabolic parameter in the sense that the map $f_{r(c)}$ has a parabolic periodic point $\alpha$ of some period $p$. A copy is called {\it primitive} if $(f^p_{r(c)})'(\alpha)=1$; in this case, $r(c)$ is called a {\it primitive root}.

 A detailed discussion of the local dynamics near a parabolic point can be found in \cite{Mil}. Let us summarize some of the relevant facts below. Fix a primitive root $r(c)\in\RR$, and let $p$ be the period of the parabolic orbit of $f\equiv f_{r(c)}$. Denote $B\equiv B_{r(c)}$ the {\it parabolic basin} of $f$, which is the collection of all points $z\in \CC$ whose orbits are attracted to the parabolic orbit. Each of the connected components of $B$ is also cauliflower-shaped. The critical point $0$ lies in $B$; let us denote $B_0$ the connected component of $B$ which contains $0$. Let $\alpha$ be the necessarily unique point of the parabolic orbit which lies in the boundary of $B_0$.

 There exist two $\RR$-symmetric topological disks $P^A$, $P^R$ known as {\it attracting and repelling petals of $\alpha$} respectively, such that:
 \begin{enumerate}
   \item $P^A\cup P^R$ form a punctured neighborhood of $\alpha$;
 \item $P^A\subset B_0$ and $\alpha\ni\partial P^A$; the iterate $f^p$  univalently maps $P^A$ into $ P^A$;
 \item for each $z\in B$ there is an iterate $f^m(z)\in P^A$; all orbits of $f^p$ in $P^A$ converge to $\alpha$ uniformly on compact subsets;
 \item $\alpha\in\partial P^R$;
 \item the local inverse branch  $f^{-p}$ of $f^{p}$ which fixes $\alpha$ univalently extends to $P^R$ and maps it into $P^R$;
 \item all orbits of $f^{-p}$ in $P^R$ converge to $\alpha$ uniformly on compact subsets;
 \item every inverse orbit of $f$ which converges to $\alpha$ intersects $P^R$;
 \item the quotient Riemann surfaces  $C^A\equiv P^A/f^p$ and $C^R\equiv P^R/f^p$ are conformally isomorphic to the bi-infinite cylinder $\CC/\ZZ$.
   \end{enumerate}
 Consider an $\RR$-symmetric conformal isomorphism $\psi_A:C^A\mapsto \CC/\ZZ$. Its lift
 $\Psi_A:P^A\to \CC$ transforms $f^p$ into the unit translation
 $$\Psi_A(f^p(z))=\Psi_A(z)+1.$$
 We call it an {\it attracting Fatou coordinate}; by Liuoville's theorem, it is defined uniquely up to an additive constant. A {\it repelling Fatou coordinate} $\Psi_R$ is defined in a similar fashion for $P^R$.

 The Douady-Lavaurs {\it parabolic implosion} theory \cite{Douady-Lavaurs} describes, in particular, what happens with the (filled) Julia set of the map $f_{r(c)}$ under a small perturbation of the parameter $r(c)\mapsto r(c)+\eps$ for $\eps>0$.
We summarize the relevant facts about parabolic implosion as follows:
 \begin{theorem}\label{t:limits1} Let $r\in\RR$ be a root of a primitive small copy of $\cM$.  There is a continuous map $\tau(0, \eps_0] \to \TT = \CC / \ZZ$ called the phase map, with a lift $\wtl{\tau}:(0,\eps_0] \to \RR$ tending to $-\infty$ as $\eps \to 0$ and an injective map $\theta \to L_{\theta}$ from $\TT$ to the set of non-empty compact subsets of $K(f_r)$ so that the following holds:
       \begin{itemize}
       \item we have $J(f_r) \subsetneq L_{\theta} \subsetneq K(f_r)$;
       \item furthermore,
         $$\underset{\tau(\eps)\to \theta}{\overline{\lim}}K(f_{\eps})\subset L_\theta\text{ and }\underset{\tau(\eps)\to \theta}{\underline{\lim}}J(f_{\eps})\supset \partial L_\theta.$$
       \end{itemize}
       Moreover, for each $\theta$ as above, there exists an analytic map $g_\theta$ called the Douady-Lavaurs map, which is defined on the basin of attraction of the parabolic orbit of $f_r$ which has the following properties:
       \begin{itemize}
       \item in the Fatou coordinates of $f_r$, the map $g_\theta$ becomes a translation by $\theta$:
         $$\Psi_R \circ g_\theta\circ (\Psi_A)^{-1}(z)=z+\theta;$$
       \item for  all $\eps\in(0,\eps_0]$, there exist integers $k(\eps)$ so that
         $$f^{k(\eps)}_{\eps}\underset{\tau(\eps)\to \theta}{\longrightarrow} g_\theta$$
         uniformly on compact sets (so $g_\theta$ is a part of the {\it geometric limit} of the dynamics of $f_{\eps}$ as $\tau(\eps)\to\theta$);
       \item $g_\theta$ commutes with $f_r$;
         \item the set $L_\theta$ is the non-escaping set of the dynamics generated by the pair $<f_r,g_\theta>$: it consists of the points with bounded orbits.

         \end{itemize}
       
\end{theorem}

\begin{figure}[ht]
\centerline{\includegraphics[width=1.3\textwidth]{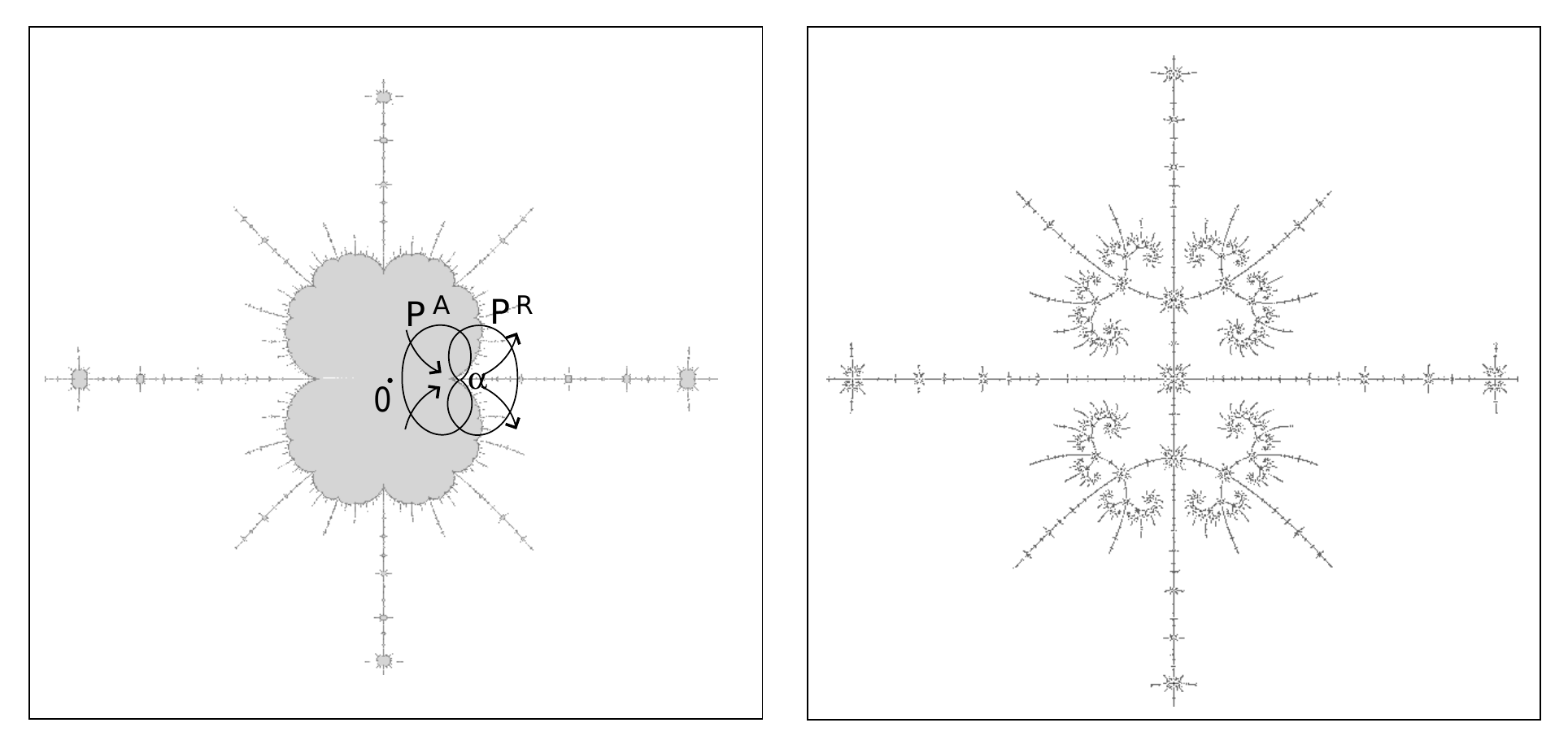}}
\caption{\label{fig:parab} Left: a zoom into the filled Julia set of $f_c$ for the parabolic parameter $c=-1.75$ near $0$. The parabolic point $\alpha$ has period $3$.
  The arrows indicate the action of $f^3_{-1.75}$ in the vicinity of $\alpha$. Right: an illustration of parabolic implosion; a zoom into the Julia set for $c=-1.75+\eps$ for a small value of $\eps>0$.}
\end{figure}

 We refer the reader to the discussion in \cite{Yam-bounds} for the following claim:
 \begin{prop}
   \label{limit-theta}Let $r\neq 1/4$ be a primitive root parameter. 
   There exists an infinite sequence of angles $\{\theta_j\}$ such that:
   \begin{itemize}
   \item $\theta_j\to\theta_\infty$ for some $\theta_\infty\in\RR/\ZZ$;
    \item for each $i\leq \infty$ there is a sequence of primitive roots $r_k^i\in\RR$ with
   $$r_k^i\underset{i\to\infty}{\searrow}r\text{, and }\tau(r_k^i-r)\to\theta_i.$$
\end{itemize}
 \end{prop}
 \begin{proof}[Outline of the proof]
   Consider any $\theta=\theta_\infty$ such that the two-generator dynamical system $<f_r,g_\theta>$ has a quadratic-like restriction with a parabolic fixed point with multiplier equal to $1$ (see Figure 4 in \cite{Yam-bounds} for an illustration). Then, there is a sequence of primitive small copies $\cM_i$ whose roots $r_i\to r$ and
   $$\tau(r_i-r)=\theta_i\to \theta.$$
   Repeating this constuction for each of the primitive roots $r_i$ instead of $r$, we obtain the desired roots $r^i_k$.
 \end{proof}
 
 We will make use of the following consequence of the parabolic implosion picture:
 
 \begin{theorem}\label{t:discontinuity}Let $r\neq 1/4$ be a primitive root parameter. Let $i=1,2$. Then there exist two strictly decreasing sequences $(r^i_{k})$ of primitive root parameters, two values
   $\theta^i\in[0,1)$, and two closed sets $L^{i}_{r}$  such that:
  \begin{itemize}
\item[i)] $r^i_k\searrow r$;
\item[ii)] in the notation of  Theorem~\ref{t:limits1}, we have  $\tau(r^i_k-r)\longrightarrow \theta^i$ and $L^i_r\equiv L_{\theta^i}$;
\item[iii)] $L^{1}_r\neq L^{2}_r$, and moreover, $\partial L^{1}_r\cap \CC\setminus L^2_r\neq\emptyset$.
\end{itemize}
\end{theorem}
 \begin{proof}
   Let $\theta_j\to\theta_\infty$ be as in \propref{limit-theta}. We will set
	$L^2_r\equiv L_{\theta_\infty}.$ To fix the ideas, let us assume that there is a decreasing subsequence
	$\theta_j\searrow \theta_\infty$ as in \propref{limit-theta} (the proof proceeds in a similar fashion 
	in the complementary case). 
	
	Let us denote $C^R$, $C^A$ the repelling and attracting Fatou cylinders of $f_r$ respectively, and
	let $$T:C^A\to C^R\text{ be given by }T(z)\equiv z+\theta_\infty\mod\ZZ.$$
	Denote $\Gamma$ the boundary of a component of the immediate parabolic basin of $f_r$. By the Maximum Principle, the 
	projection 
	$$\gamma=\Psi^R(\Gamma\cap P^R)\mod\ZZ\subset C^R$$
	is a simple closed curve on the repelling Fatou cylinder, homotopic to its equator $\RR/\ZZ$; set 
	$$\delta=T^{-1}(\gamma)\subset C^A.$$
	Standard facts of the parabolic implosion theory imply that 
	$$\delta\subset\Psi^A(\partial  L_{\theta_\infty}).$$
	Note that $\Gamma$ contains a cusp at the parabolic point. Conformal self-similarity considerations imply that cusps are dense in $\Gamma$, and hence also in $\delta$. In particular,  $\delta$ is not a circle. Thus, there exists a horizontal circle 
	$$S=\{\Im(z)=a>0\}\mod\ZZ\subset C^A$$
	which intersects $\delta$ and the escaping set $\Psi^A(\CC\setminus L_{\theta_\infty})$. Evidently, there is a 
	point $z_0$ on $S\cap\delta$ and $\alpha_0>0$ such that for any $\alpha\in(0,\alpha_0)$, 
	$$z_0-\alpha\mod\ZZ\in \Psi^A(\CC\setminus L_{\theta_\infty}).$$
	Let $\theta_j\in (\theta_\infty,\theta_\infty+\alpha_0/2)$. Setting
	$L^1_r=L_{\theta_j}$ completes the proof.

   \end{proof}
 
\begin{corollary}
  \label{cor:disc}
  Let the values $\theta^1$, $\theta^2$ be as above. For all  $n\in\NN$ sufficiently large, there exist $\delta>0$ and $z_0\in\CC$ such that the following holds. Let $\eps_i>0$ be such that
  $$|\tau(\eps_i)-\theta^i|<\delta.$$
  Set $c_i\equiv r+\eps_i$. Then
  $$J(f_{c_1})\cap \{|z-z_0|<2^{-n}\}\neq \emptyset\text{ and }K(f_{c_2})\cap \{|z-z_0|<2^{-(n-1)}\}= \emptyset.$$

\end{corollary}
\begin{proof}
  Let $D$ be a disk in $\CC\setminus L^2_r$ which intersects  $\partial L^{1}_r$. 
  By the structure theory of Douady-Lavaurs maps \cite{epstein}, repelling periodic orbits of $<f_{r},g_\theta>$ are dense in $\partial L_\theta$.
  Consider any repelling periodic point $\beta$ of $<f_r,g_{\theta^1}>\in D$.
  There is a sub-disk $W=D_t(\beta) \Subset D$ for some $t>0$ and a composition $F$ of iterates of $f_r$ and $g_{\theta^1}$ which is conformal in $W$ and such that $F^{-1}(W)\Subset W$. By Theorem~\ref{t:discontinuity}, provided $\delta$ is small enough, there is an iterate $f_{c_1}^n$ which is a sufficiently small perturbation of $F$ so that $f_{c_1}^n$ is conformal in $W$, and the inverse branch $(f_{c_1}^n|_W)^{-1}$ maps $W$ to $W'\Subset W$. The Schwarz Lemma implies that $f_{c_1}$ also has a repelling periodic point in $W$, and hence, $J(f_{c_1})\cap D\neq \emptyset$.
  
  On the other hand, shrinking $D$ if necessary, and using the same argument as above, we see that $D$ lies in the escaping set of $f_{c_2}$.
  \end{proof}

\subsection{Constructing Julia sets of prescribed complexity}

Let us begin by stating the standard lower semi-continuity property of $J_c$ and upper semi-continuity of $K_c$ (see \cite{Douady-Lavaurs}):
\begin{lemma}\label{semi-continuity}
  For any $\hat c\in\CC$ and any $\eps>0$ there exists $\delta>0$ such that
  $$\dist(J_{\hat c},J_c)<\eps\text{ and }\dist(K_c,K_{\hat c})<\eps$$
  for all $c$ such that $|c-\hat c|<\delta.$

\end{lemma}
  
Let $T(l)$ be any increasing function.  Let $(M_n)_n$ be the list of all machines with an oracle for $c$ whose running time is less than $T(l)$.  That is, when provided with a dyadic point $p$ of size $l$ as input, the machine $M^{\phi}_n(p)$ halts in less than $T(l)$ steps and outputs 0 or 1. Note that during the computation, the machine can query the oracle $\phi$ to learn, at most, $T(l)$ bits of $c$. 

Our construction can be thought of as a game between a \emph{Player} and infinitely many \emph{opponents}, which will correspond to the machines $M_n$.  The opponents try to compute $J_c$ by asking the Player to provide an oracle $\phi$ for $c$, while the Player tries to chose the bits of $c$ in such a way that none of the opponents correctly computes $J_c$. We show that the Player always has a wining strategy: it plays against each machine, one by one, asking the machine to decide a particular pixel $p$ of a certain size. The machine then runs for a while asking the Player to provide more and more bits of $c$, until it eventually halts and outputs 0 or 1. Then the Player reveals the next bit of $c$ and shows that the machine's answer is incompatible with $J_c$. The details are as follows.

We will proceed inductively. At step $n$ of the induction, we will have a parabolic parameter $c_n$, a natural number $l_n$ and a dyadic point $p_{n}$ of size $l_{n}$  such that:
\begin{enumerate}
\item\label{machines} One the following two possibilities holds:
\begin{itemize}\label{fooling}
\item  $M_{n}^{\phi}(p_{n}) = 0 $ whereas   $\dist(p_{n}, J_{c_{n}})<2^{-l_{n}}/2$,
\item  $M_{n}^{\phi}(p_{n}) = 1 $ whereas   $\dist(p_{n}, J_{c_{n}})>3\cdot 2^{-l_{n}}$.  
\end{itemize}
In other words, given an oracle for $c_n$, the machine $M_n^\phi$ cannot decide pixel $p_{n}$ of $J_{c_{n}}$ in time $T(l_n)$;
\item\label{pixels}  $\dist(J_{c_{n-1}}, J_{c_{n}})  < 2^{-3l_{n-1}}$ and $\dist(K_{c_{n}}, K_{c_{n-1}})  < 2^{-3l_{n-1}}$; 
\item\label{parameters} $|c_n-c_{n-1}|<2^{-3l_n}$.
\end{enumerate}
\bigskip

\noindent
    {\bf Base of induction.}    We start by considering the parameter $r=-1.75$, which is the primitive root of the period 3 copy of $\M$.  For $i\in \{1,2\}$ let $c^{i}_{k}(r)$ be the two sequences given by Theorem \ref{t:discontinuity}.
     By Corollary~\ref{cor:disc}, there  exists $l_{1}\in \NN$ and $k_{0}$ such that  for all $k\geq k_{0}$, $$\dist_H(J_{c^{1}_{k}(r)}, J_{c^{2}_{k}(r)}) >  10 \cdot 2^{-l_{1}}.$$ Moreover, for such a $k$ there is a dyadic point $p_{1}$ of size $l_{1}$ such that
$$
\dist(p_{1}, J_{c^{1}_{k}(r)}) < 2^{-l_{1}}/10 \quad  \text{ and } \quad     \dist(p_{1},J_{c^{2}_{k}(r)}) > 8 \cdot 2^{-l_{1}}. 
$$
 
We let the machine $M_1^\phi$ compute $J_{c}$ at $p_{1}$ with precision $2^{-l_1}$, giving it $r=-1.75$ as the parameter.  If the machine outputs $s\in\{0,1\}$, we set $c_{1}\equiv c^{s+1}_{k}$ where $k\in \NN$ is chosen large enough so that  the condition 
 $$|c^{s+1}_k+1.75|<2^{-T(l_1)}$$ 
 holds as well.   Note that  in the running time $T(l_{1})$ the machine $M^\phi_{1}$ cannot tell the difference between parameter $r$ and parameter $c_{1}=c^{s+1}_{k}$, and therefore it will halt and output the same answer for both parameters. This guarantees condition (\ref{machines}) to hold. 
 
 \bigskip
 
 \noindent
    {\bf Step of induction.} Assume $c_{n}$ has been constructed.   By  Theorem \ref{t:discontinuity}  and Corollary~\ref{cor:disc} again,  there exists two sequences $c^{i}_{k}$ of primitive root parameters which converge to $c_{n}$ from the right for which  the corresponding sequences of Julia sets $J_{c^{i}_{k}}$ have different Hausdorff limits, and in particular, there is a dyadic point $p_{n+1}$ of size $l_{n+1}$ and an integer $k\in\NN$ such that
$$
\dist(p_{n+1}, J_{c^{1}_{k}}) < 2^{-l_{n+1}}/10 \quad  \text{ and } \quad     \dist(p_{n+1},J_{c^{2}_{k}}) > 8 \cdot 2^{-l_{n+1}}. 
$$

 We let the machine $M_{n+1}^\phi$ compute $J_{c}$ at $p_{n+1}$ with precision $2^{-l_{n+1}}$, giving it $c_{n}$ as the parameter.  If the machine outputs $s$, we can guarantee condition (\ref{machines}) by setting $c_{n+1}\equiv c^{s+1}_{k}$ with $k$ large enough so that $$|c^{s+1}_k - c_{n}|<2^{-T(l_{n+1})}.$$ Once again, in the running time $T(l_{n+1})$ the machine $M^\phi_{n+1}$ cannot tell the difference between $c_{n}$ and $c_{n+1}$, and therefore it will halt and output the same answer for both parameters.  We can clearly chose $k$ large enough so that condition (\ref{parameters}) is verified as well.

%
%
 %

 Finally, choosing $k$ so as to ensure that condition (\ref{pixels}) holds is possible by Lemma~\ref{semi-continuity}. Note, that it guarantees that (up to a very small error)
 pixels in the picture of the Julia set that we have already created at step $n-1$ will remain in the picture of the Julia set created at step $n$, and the same is true for pixels in the basin of infinity.  
 
%
%


We now let $c_{\infty}=\lim_nc_n$ and claim that $J_{c_\infty}$ has the required properties. Indeed, condition (\ref{fooling}) ensures that for every $n$, there is a pixel $p_n$ of size $l_n$ that machine $M_n$ fails to decide correctly for $J_{c_n}$,  and condition (\ref{pixels}) guarantees that the same holds for $J_{c_{\infty}}$.

%
%
%
%
%
\bigskip

\newpage

\bibliographystyle{amsalpha}
\bibliography{bib_RY}

\end{document}